\documentclass[a4paper,12pt]{article}
\usepackage[left=1in, right=1in, top=1in, bottom=1in]{geometry}
\usepackage[font=small,labelfont=bf,width=0.8\textwidth,skip=16pt]{caption}
\usepackage{graphicx}
\usepackage{here} 
\usepackage{bbm} 
\usepackage{color,latexsym,amsfonts,amssymb,amsmath,mathbbol,bm,hyperref}
\usepackage{amsthm}
\usepackage{verbatim}
\usepackage{mathrsfs}
\usepackage{fancyhdr}
\usepackage{bbold}

\numberwithin{figure}{section}
\numberwithin{table}{section}
\numberwithin{equation}{section}

\newcommand{\E}{\mathbb E}

\newcommand{\cL}{\mathcal L}

\newcommand{\Z}{\mathbb Z}

\usepackage{amsthm}
\newtheoremstyle{mythm}{18pt}{0pt}{\itshape}{}{\bfseries}{.}{12pt}{}
\newtheoremstyle{mydefn}{18pt}{0pt}{}{}{\bfseries}{.}{12pt}{}
\theoremstyle{mythm}
\newtheorem{theorem}{Theorem}[section]
\newtheorem{lemma}[theorem]{Lemma}
\newtheorem{proposition}[theorem]{Proposition}

\newtheorem{corollary}[theorem]{Corollary}
\theoremstyle{mydefn}
\newtheorem{remark}[theorem]{Remark}

\def\Pr{\mathbb{P}}

\newcommand{\NB}{{\rm NB}}
\newcommand{\Var}{{\rm Var}}

\def\Ref#1{(\ref{#1})}

\parskip = \baselineskip
\newcommand{\eqs}{\begin{eqnarray*}}
\newcommand{\ens}{\end{eqnarray*}}

\def\l{\lambda}

\newcommand{\beas}{\begin{eqnarray*}}
\newcommand{\enas}{\end{eqnarray*}}

\newcommand{\eqa}{\begin{eqnarray}}
\newcommand{\ena}{\end{eqnarray}}
\newcommand{\eq}{\begin{equation}}
\newcommand{\en}{\end{equation}}

\def\ind{{\bf 1}}

\def\l{\lambda}

\def\ignore#1{}

\def\Po{{\rm Po}}

\def\l{\lambda}

\def\ed{\stackrel d=}
\def\ton{^{[0,n]}}

\makeatletter
\renewcommand\theequation{\thesection.\@arabic\c@equation}
\makeatother

\usepackage{color,graphicx,mathbbol, enumerate}

\def\ind{\bm{1}}  
\def\E{\mathbb{E}} 

\def\Pr{\mathbb{P}} 
\def\Z{\mathbb{Z}}

\def\Ref#1{(\ref{#1})}

\newcommand{\Nb}{\text{NB}}

\allowdisplaybreaks
\begin{document}
\parindent 0cm
\parskip .5cm
\title{Improving approximation error bounds via truncation}
\author{H.~L.~Gan\footnote{Mathematics Department,
Northwestern University,
2033 Sheridan Road,
Evanston, IL 60208, USA. E-mail: ganhl@math.northwestern.edu}}
\maketitle
\begin{abstract}
One aspect of Poisson approximation is that the support of the random variable of interest is often finite while the support of the Poisson distribution is not. In this paper we will remedy this by examining truncated negative binomial (of which Poisson is a special limiting case) approximation, so as to match the two supports of both distributions, and show that this will lead to improvements in the error bounds of the approximation.
\end{abstract}

\section{Introduction}
The classical law of small numbers states that the binomial distribution Bi$(n,p)$ converges to the Poisson distribution Po$(\l)$ as $n \to \infty$ and $p = \l/n$. This naturally leads to the Poisson distribution being used as an approximation for the binomial distribution when $n$ is large and $p$ (or $1-p$) is small. Evaluation of the accuracy of such an approximation has been studied in depth over the years, see the introduction of \cite{BHJ92} for a short discussion. There is however one property of the Poisson distribution that is somewhat not ideal for such an approximation. That is, the Poisson distribution has support upon all non-negative integers, whereas the binomial distribution has support on non-negative integers less than or equal to $n$. As a simple correction, we could truncate the Poisson distribution and consider using the Poisson distribution \emph{conditional upon being less than or equal to $n$}.

This paper will focus on Stein's method for negative binomial approximation (of which Poisson approximation is special limiting case), truncated to be conditional upon being less than or equal to $n$. We will formulate the framework for truncated negative binomial approximation, defining an appropriate Stein equation and deriving bounds for the relevant Stein factors. Our main result is that given the random variable of interest has a finite support, a Stein's method bound for truncated distribution approximation will yield a strictly better bound than the usual approximation. In section 2 we will set up the framework for truncated distribution approximation with Stein's method, and state the main results. The proofs will follow in section 3.

\section{Main results}
Negative binomial approximation with Stein's method has been well studied in the past, see for example~\cite{BP99} for approximation bounds in total variation distance and~\cite{BGX16} for the Wasserstein distance. We will use the following definition of the negative binomial distribution. A random variable $Z$ follows the negative binomial distribution $\Nb(r,p)$ if it has mass function given by
\[ P(Z = k) = \frac{\Gamma(r+k)}{\Gamma(r) k!}(1-p)^r p^k, \ \ k \in \mathbb{Z}_+ := \{0,1,...\}; r > 0, 0 < p < 1.\]
Let $Z\ton$ denote a random variable conditioned upon being less than or equal to $n$, that is $\cL(Z\ton)= \cL(Z | Z \leq n)$, and analogously define $\NB\ton(r,p)$. Using similar ideas to~\cite{GX15}, a Stein identity for the truncated negative binomial distribution is as follows.
\begin{lemma}
For a non-negative integer $n$, $W \ed \NB\ton(r,p)$ if and only if for for all bounded functions $g$ on $\mathbb{Z}\ton :=\{0,1,\ldots, n\}$
\begin{align} \E\left[ p(r+W) g(W+1) \ind_{W < n} - W g(W) \right] = 0.\label{steinid}\end{align}
\end{lemma}
\begin{proof}
The forward implication is true by direct computation and the backward implication can be checked by considering test functions $g_i(x) = \ind_{\{i\}}(x)$ and solving for the probabilities.
\end{proof}
We define the solution to the Stein equation $g_f$ to be the function that for a given function $f$ satisfies 
\begin{align} p(r+i)g_f(i+1) \ind_{i<n} - i g_f(i) = f(i) - \NB\ton(r,p) \{f\}, \ \  f \in \mathcal{F},\label{steineq}\end{align}
for some rich enough family of functions $\mathcal F$ and where $\NB\ton(r,p) \{f\} = \E f(Z)$ for $Z \ed \NB\ton(r,p)$. Then for a random variable $W$ on $\mathbb Z\ton$,
\begin{align} \E f(W) - \NB\ton(r,p)\{f\} = \E \left[ p(r+W)g_f(W+1) \ind_{W<n} - W g_f(W)\right]. \label{steineq2}\end{align}
By choosing $\mathcal F = \mathcal F_{TV} := \{ f : f(x) = \ind_A(x) , A \subset \mathbb{Z}\ton\}$, if we can find a uniform bound for the right hand side of \Ref{steineq2} over all $f \in \mathcal F_{TV}$ this corresponds to a bound upon the total variation distance between $\cL(W)$ and $\NB\ton(r,p)$ where the total variation distance between two distributions $P$ and $Q$ with support on $\mathbb{Z}\ton$ is defined as
\[ d_{TV}(P,Q) = \sup_{f \in \mathcal{F}_{TV}} \left| \int f dP - \int f dQ\right| = \frac12 \sum_{i=0}^n \left| P\{i\} - Q\{i\} \right|.\]
To apply Stein's method, a key ingredient is to have bounds of the right order upon the first difference of $g_f$. That is we need 
\[ G_2\ton := \sup_{f \in \mathcal{F}_{TV}} \sup_{i \in \Z\ton} | g_f(i+1) - g_f(i)|.\]
\begin{lemma}\label{steinfactor}
For any $n$,
\begin{align}
G_2\ton \leq \frac{1- \pi_0}{pr} = \frac{\Pr(1 \leq Z \leq n)}{pr\Pr(Z \leq n)}, \label{factorbound}
\end{align}
where $Z \ed \NB(r,p)$ and $\pi_0 = \Pr(Z\ton = 0) = \Pr(Z=0) / \Pr(Z \leq n)$. Furthermore, this bound is sharp and increasing in $n$.
\end{lemma}
If we take $p \to 0$, $r \to \infty$ and $pr \to \l$, the negative binomial distribution limits to the Poisson distribution, hence we also immediately get the following.
\begin{corollary}
For Poisson approximation, for any $n$
\begin{align} G_2\ton = \frac{1-\pi_0}{\l} = \frac{\Pr(1 \leq Z \leq n)}{\l\Pr(Z \leq n)},\end{align}
where $Z \ed \Po(\l)$ and $\pi_0 = \Pr(Z\ton = 0) = \Pr(Z=0) / \Pr(Z \leq n)$. Furthermore, this bound is sharp and increasing in $n$.
\end{corollary}
\begin{remark}
Note that as $n \to \infty$ the above yield exactly the same bound as the existing bounds for the negative binomial Stein factor $\frac{1-p^r}{rp}$ in Lemma 5 of~\cite{BP99},  and similarly for the Poisson equivalent $\frac{1-e^{-\l}}{\l}$ in Corollary~2.12 in~\cite{BX01}.
\end{remark}
A result of Lemma \ref{steinfactor} yields that for a bounded random variable, given there exists a Stein's method type bound for (unconditional) negative binomial or Poisson approximation, a truncated approximation will yield a better error bound.
\begin{theorem}\label{mainthm}
Let $W$ be a random variable that has support on $\mathbb Z\ton$. If it can be shown that there exists a constant $C$ such that
\begin{align*}
\E\left[ p(r+W) g(W+1) - W g(W) \right] \leq C G_2^{[0,\infty)},
\end{align*}
then
\begin{align}
d_{TV}(\cL(W), \NB\ton(r,p)) \leq C G_2\ton \leq  C G_2^{[0,\infty)}.
\end{align}
\end{theorem}
It should be noted that one of the uses for Stein's method is to calculate explicit bounds for rates of convergence. The approach in this paper is therefore not useful in this regard given our reference distribution also depends upon $n$. However if the purpose is strictly for obtaining the most accurate approximation, then Theorem~\ref{mainthm} shows that this approach leads to an improvement in existing error bounds.

Although Theorem~\ref{mainthm} shows that for finite random variables the approximation bounds will be improved by using truncated versions of the approximating distribution, the difference is going to be negligible when $n$ is much larger than $\lambda$. As an example, for the law of small numbers $n$ is typically very large, and $p$ is small, meaning that $\l$ is much smaller than $n$ and the effect of truncating is essentially unnoticeable. In the following we present a toy example where $n$ and $\l$ could be similar, hence truncation would offer noticeable improvements. 

Consider a machine that will develop a fault on any given day independently with probability $p$. When a fault occurs, the machine takes $R$ days to repair before it returns to service. Over a period of $N$ days, let $W$ denote the number of times a fault occurs assuming it starts in a working state. In this case an appropriate choice for $n$ would be $n = \lfloor{N/R}\rfloor +1$, and depending upon the choice of parameters, $\l$ could be close to $n$.
\begin{proposition}\label{example}
In the above set up, if we set $\l = \E(W)$, then,
\[ d_{TV}(\mathcal L(W), \Po\ton(\l)) \leq \l G_2\ton \left[ 1- \frac{\Var \ W}{\l} \right],\]
and this bound is of order $p$.
\end{proposition}
\section{Proofs}
\begin{proof}[Proof of Lemma~\ref{steinfactor}.]
If we set $g_f(i) = h_f(i) - h_f(i-1)$, we can rewrite the Stein identity~\Ref{steinid} in terms of a generator
\begin{align}
\mathcal A\ton h_f(i) = p(r+i)[h_f(i+1) - h_f(i)]\ind_{i < n} + i[h_f(i-1) - h_f(i)].\label{generator}
\end{align}
This is the generator of a birth-death process on non-negative integers with birth and death rates respectively,
\begin{align}
\alpha_i &= 
\begin{cases}
	p(r+i) & i < n\\
	0 & i \geq n
\end{cases},\label{rates}\\
\beta_i &= i,\ \ \  \forall i. \notag 
\end{align}
The stationary distribution associated to the generator~\Ref{generator} can be shown to be $\NB\ton(r,p)$. Let $\pi_i = \Pr(Z\ton = i)$ where $Z\ton \ed \NB\ton(r,p)$. Noting that these rates satisfy condition (C4) in Lemma~2.4 of~\cite{BX01}, we can therefore use Theorem~2.10 from~\cite{BX01}. Hence
\begin{align*}
&\sup_{f \in \mathcal{F}_{TV}} | \Delta g_f (i)| = \frac{\pi_{i+1} + \ldots + \pi_n}{p(r+i)} + \frac{\pi_0 + \ldots + \pi_{i-1}}{i}\\
	&=\frac{1-\pi_0}{p(r+i)} + \sum_{j=0}^{i-1} \left( \frac{\pi_j}{i} - \frac{\pi_{j+1}}{p(r+i)}\right)\\
	&= \frac{1-\pi_0}{pr} + (1-\pi_0) \left( \frac{-i}{pr(r+i)} \right) + \sum_{j=0}^{i-1} \left( \frac{\pi_j}{i} - \frac{\pi_{j+1}}{p(r+i)}\right)\\
	&= \frac{1-\pi_0}{pr} + (\pi_{i+1} + \ldots + \pi_m) \left( \frac{-i}{pr(r+i)} \right) +  \sum_{j=0}^{i-1} \left( \frac{\pi_j}{i} - \pi_{j+1} \left( \frac{1}{p(r+i)} + \frac{i}{pr(r+i)}\right) \right)\\
	&\leq \frac{1-\pi_0}{pr} +   \sum_{j=0}^{i-1} \left( \frac{\pi_j}{i} - \frac{\pi_{j+1}}{pr}  \right)\\
	&= \frac{1-\pi_0}{pr} + \sum_{j=0}^{i-1} \pi_j \left( \frac{1}{i} - \frac{p(r+j)}{(j+1)pr}\right)\\
	&\leq \frac{1-\pi_0}{pr} + \sum_{j=0}^{i-1}\pi_j \left( \frac{1}{j+1} - \frac{(r+j)}{(j+1)r}\right)\\
	&\leq \frac{1-\pi_0}{pr}.
\end{align*}
To show that the bound is sharp, note that it can be shown (see~\cite{BB92}) that 
\[ h_f(i)= -\int_0^\infty\left[ \E f(Z_i\ton(t)) - \NB\ton(r,p)\{ f\}\right] dt,\]
where $Z_i\ton(t)$ is a birth-death process following generator~\Ref{generator} with $Z_i\ton(0) = i$. Let $\tau_{0,1}$ denote the first transition time for the process $Z_0\ton(t)$, then
\begin{align}
h_f(0) &= -  \E \int_0^{\tau_{0,1}}\left[ f(Z_0\ton(t)) - \NB\ton(r,p)\{ f\}\right] dt\notag\\
	&\ \ \ \ \  -  \E \int_{\tau_{0,1}}^\infty \left[f(Z_0\ton(t)) - \NB\ton(r,p)\{ f\}\right] dt. \label{decomp}
\end{align}
Set $f(x) = \ind_{\{0\}}(x)$ and noting the Markov property means the second term of the above is equal to $h_f(1)$, then
\begin{align*}
|g_f(1)| = |h_f(1) - h_f(0)| = \E \tau_{0,1} \cdot  (1 - \pi_0) = \frac{1-\pi_0}{pr}.
\end{align*}
We can assume that $g_f(0) = 0$, and hence this shows that the bound \Ref{factorbound} can be attained and is therefore sharp. To show that the bound is increasing in $n$, simply note that $\pi_0$ is decreasing in $n$.
\end{proof}
%

\begin{proof}[Proof of Theorem~\ref{mainthm}]
In the unconditional setting we need to bound 
\[\E\left[ p(r+W) g(W+1) - W g(W) \right].\]
If $W \leq n$, the only difference with this identity compared to our identity for the truncated negative binomial is $\E[ p(r+W) g(W+1) \ind_{W = n}]$. Therefore it would be sufficient to show that $g(n+1)=0$, use our Stein factor for truncated negative binomial in \Ref{factorbound} and the final conclusion follows from Lemma~\ref{steinfactor}.

Recall that we can set $g_f(i) = h_f(i) - h_f(i-1)$.
Given the transition rates in \Ref{rates}, the process $Z_{n+1}\ton(t)$ can only transition downwards with rate $n+1$. Let $\tau \ed \exp(n+1)$, then in a similar fashion to \Ref{decomp}, using the Markov property
\begin{align*}
h_f(n+1) - h_f(n) = - \E \int_0^\tau \left[ \E f(Z_{n+1}\ton(t)) - \NB\ton(r,p)\{ f\}\right] dt =- \E \tau \cdot (f(n+1) - \pi_{n+1}).
\end{align*}
Since the support of $f$ and $\NB\ton(r,p)$ is $\mathbb{Z}\ton$, $f(n+1) = \pi_{n+1}=0$, and hence $g_f(n+1)=0$.
\end{proof}

\begin{proof}[Proof of Proposition~\ref{example}]
For $i \in \{1, 2, \ldots, N\}$ let $I_i = 1$ if a fault occurs on day $i$, otherwise $I_i=0$, then $W = \sum_{i=1}^N I_i$. It is straightforward to see that the $\{I_i\}$ are negatively related, and Corollary~2.C.2 from~\cite{BHJ92} states that
\[ d_{TV}(\mathcal L(W), \Po(\lambda)) \leq (1-e^{-\lambda})\left( 1 - \frac{\Var \ W}{\l} \right).\]
The result now follows by using Theorem~\ref{mainthm} to swap Stein factor into the truncated setting. 

The exact quantities for $\lambda$ and $\Var \ W$ are simple enough but tedious to calculate for general $N$, $p$ and $R$, so we choose to leave out the explicit formulas and focus on the order of the approximation. To show that the bound is of order $p$, note that for fixed $n$, the probability of one fault is of order $p$ and the probability of two faults is of order $p^2$ and so forth. This implies that $\frac{\Var\ W}{\E W}$ is of order $1 + O(p)$ and this gives the final result.
\end{proof}

\end{document}